\documentclass[A4paper,12pt]{article}
\usepackage{latexsym}
\usepackage{mathrsfs}
\usepackage{amssymb}
\usepackage{amscd}
\usepackage[dvips]{graphicx}                  

\newtheorem{theorem}{Theorem}

\newtheorem{lemma}{Lemma}

\newenvironment{proof}[1][Proof]{\textbf{#1.} }{\ \rule{0.5em}{0.5em}}
\date{}
\long\def\symbolfootnote[#1]#2{\begingroup%
	\def\thefootnote{$\;$}\footnote[#1]{$^*$#2}\endgroup}
\begin{document}
	
	\title{A short mathematical proof of \\the Frankiewicz - Kunen theorem}
	\author{Ryszard Frankiewicz and Joanna Jureczko}
\maketitle

\symbolfootnote[2]{Mathematics Subject Classification: 5Primary 03C25, 03E55, 03C20, 54E52). 
	
	\hspace{0.2cm}
	}

\begin{abstract}
In this paper there is proved without any metamathematical techniques that the existence of precipitous ideals immediately follows from the existence of Kuratowski partitions.
\end{abstract}

\maketitle

\section{Introduction}

In 1935 K. Kuratowski in \cite{KK} posed the problem whether a function $f \colon X \to Y$, (where $X$ is completely metrizable and $Y$ is metrizable), such that each preimage of an open set of $Y$ has the Baire property, is continuous apart from a meager set.

In \cite{EFK} there is shown the equivalence of this problem with the problem of the existence of partitions of completely metrizable spaces into meager sets with the property that the union of each subfamily of this partition has the Baire property. Such a partition is called a \textit{Kuratowski partition}, (see the next section for a formal definition).

In the 70's of the last century R. H. Solovay and L. Bukovsk\'y independently proved non-existence of Kuratowski partitions of a unit interval $[0,1]$ for measure and category using forcing methods (and the generic ultrapower), but Bukovsk\'y' proof, see \cite{LB}, is shorter and less complicated than Solovay's (unpublished results).

With  a Kuratowski partition there is associated, in a natural way, an ideal which is called in \cite{JW} a \textit{$K$-ideal}, (see the next section for a formal definition).
It can be supposed that from a structure of such a $K$-ideal one can decode full information about a Kuratowski partition of a given space. Unfortunately, it is not so because, as was shown in \cite{JW}, the structure of such an ideal can be almost arbitrary, i.e. it can be a Fr\'echet ideal, so in the presence of 
\cite[Lemma 35.9, p. 440]{TJ} it is not precipitous if $\kappa$ is regular, (see the next section for the definition of precipitous ideals).
Moreover, there is shown in \cite{JW} that for each measurable cardinal $\kappa$, a $\kappa$-complete ideal can be represented by some $K$-ideal.
Thus, for obtaining a Kuratowski partition from a $K$-ideal we need full information about the space in which the ideal is considered.

The natural question is about assumptions under which the existence of Kuratowski partitions and precipitous ideals are related, i.e. under which assumptions a $K$-ideal  is precipitous.
In \cite{FK} there is shown, among others, that ZFC + there is a Kuratowski partition is consistent, then ZFC + there is a measurable cardinal is consistent as well, using  forcing methods in the proof (i. e. a model of the G-generic ultrapower in Keisler sense), (see \cite[sec. 6.4]{CK} and \cite{TJ} for details) and the Banach Localization Theorem, (see \cite[p. 82]{KK1}). 

The main goal of this paper is to show that the existence of a Kuratowski partition $\mathcal{F}$ in a metric Baire space implies the existence of a precipitous $K$-ideal associated with $\mathcal{F}$. 
As we will show, using the Banach Localization Theorem we do not need to use forcing techniques, which is the main idea in the presented proof. 
In the contrary to enlarge spaces in proofs in \cite{JW} we will reduce a space and use only some combinatorial properties of precipitous ideals,  (reminded in the next section) and the Baire Category Theorem. 

\section{Definitions and previous results}
Let $X$ be a topological space and $\kappa$ be a cardinal, ($\kappa$ may be assumed as a regular cardinal).

We say that a family $\mathcal{F}$ consisted of meager sets of $X$ such that $\bigcup\mathcal{F} = X$ is a \textit{Kuratowski partition} if $\bigcup \mathcal{F}'$ has the Baire property for any subfamily $\mathcal{F'}\subset \mathcal{F}$ .

With any Kuratowski partition 
$\mathcal{F} = \{F_\alpha \colon \alpha < \kappa\}$, indexed by $\kappa$,  one may associate an ideal 
$$I_\mathcal{F} = \{A \subset \kappa \colon \bigcup_{\alpha \in A} F_\alpha \textrm{ is meager}\}$$
which is called \textit{a $K$-ideal}, (see \cite{JW}).

Let $I$ be an ideal on $\kappa$ and let $S$ be a set with positive measure, i.e. $S \in P(\kappa) \setminus I = I^+$. An \textit{$I$-partition} of $S$ is a maximal family $W$ of subsets of $S$ of positive measure such that $A \cap B \in I$ for all distinct $A, B \in W$.
An $I$-partition $W_1$ of $S$ is a \textit{refinement} of an $I$-partition $W_2$ of $S$, ($W_1 \leq W_2$), if each $A \in W_1$ is a subset of some $B\in W_2$.

A \textit{functional} on $S$ is a collection $\Phi$ of functions such that $W_\Phi = \{dom(f) \colon f \in \Phi\}$ is an $I$-partition of $S$ and $dom(f) \not = dom(g)$, whenever $f \not=g \in \Phi$.
\\
We define $\Phi < \Psi$ if
\\
(i) each $f \in \Phi \cup \Psi$ is a function into the ordinals;
\\
(ii) $W_\Phi \leq W_\Psi$;
\\
(iii) if $f \in \Phi$ and $g \in \Psi$ are such that $dom(f) \subseteq dom(g)$, then $f(x) < g(x)$ for all $x \in dom(f)$.

If  $I$ is a $\kappa$-complete ideal on $\kappa$ containing singletons then $I$ is \textit{precipitous} iff whenever $S$ is a set of a positive measure and $\{W_n \colon n < \omega\}$ is a sequence of  $I$-partitions of $S$ such that 
$W_0 \geq W_1\geq ... \geq W_n \geq ...$
then there exists a sequence of sets
$X_0 \supseteq X_1\supseteq ... \supseteq X_n \supseteq ...$
such that $X_n \in W_n$ for each $n\in \omega$ and $\bigcap_{n=0}^{\infty} X_n \not = \emptyset$, (see also \cite[p. 438-439]{TJ}).
\\

In the proof of the next theorem we will need the following characterization of precipitous ideals, (see \cite[Lemma 35.8, p. 439]{TJ}).

\begin{lemma}[\cite{TJ}] The following are equivalent
	
	(i) $I$ is precipitous;
	
	(ii) For no $S$ of a positive measure is there a sequence of functionals on $S$ such that $\Phi_0>\Phi_1>...>\Phi_n>...\ .$
\end{lemma}

\section{The main result}

\begin{theorem}
	Let $X$ be a metric Baire space with a Kuratowski partition $\mathcal{F}$. Then there exists an open set $U \subset X$ such that a $K$-ideal $I_{\mathcal{F}\cap U}$ is precipitous.
\end{theorem}

\begin{proof}
Let $\kappa = \min \{|\mathcal{F}| \colon \mathcal{F} \textrm{ is a Kuratowski partition of } X\}$.
Let $$\mathcal{F} = \{F_\alpha \colon \alpha < \kappa\}$$ be a fixed Kuratowski partition of $X$ and 
let \\$FN(\kappa) = \{f \in {^X}\kappa \colon \exists_{\mathcal{U}_f} \textrm{ a family of }  \textrm{open disjoint sets which is }  \textrm{dense in } X  \\\textrm{ and } \forall_{F_\alpha \in \mathcal{F}} \forall_{U \in \mathcal{U}_f} f \textrm{ is constant on }  F_\alpha\cap U\}$, \\ (compare \cite[proof of Theorem 3.3]{FK}).

We claim that there exists $f \in FN(\kappa)$ and $U \in \mathcal{U}_f$ such that 
$$I_{\mathcal{F}\cap U} = \{A \subset \kappa \colon \bigcup_{\alpha \in A} F_\alpha\cap U \textrm{ is meager, } F_\alpha \in \mathcal{F}\}$$ is precipitous.

Suppose not. 
For any $f \in FN(\kappa)$ accept the following notation
$$\mathcal{F}_f = \{F_\alpha \cap U \colon F_\alpha \in \mathcal{F}, U \in \mathcal{U}_f\},$$
i.e. $\mathcal{F}_f$ is a refinement of $\mathcal{F}$ for preimages of $f \in FN(\kappa)$.

Then by Lemma 2.1 for each $F_\alpha \cap U \in \mathcal{F}_f$ there exists a sequence of functionals 
$\Phi^f_{0, F_\alpha \cap U} > \Phi^f_{1,F_\alpha \cap U}> ... $ on some set $S^f_{F_\alpha \cap U} \in P(\kappa)\setminus I_{\mathcal{F}\cap U}$.

Let $W^f_{i, F_\alpha \cap U}$ be an  $I_{\mathcal{F}\cap U}$-partition corresponding with $\Phi^f_{i,F_\alpha \cap U}$,  $i = 0, 1, ...\ $.
Then since $I_{\mathcal{F}\cap U}$ is not precipitous $\bigcap^{\infty}_{i=0} X^f_{i,F_\alpha \cap U} = \emptyset$ for all $X^f_{i, F_\alpha \cap U} \in W^f_{i, F_\alpha \cap U}$, $ i = 0, 1, ...\ .$ 
Each  $X^f_{i, F_\alpha \cap U}$ is a domain of some function $h^f_{i, F_\alpha \cap U} \in \Phi^f_{i, F_\alpha \cap U}$.
For each $h^{f}_{i, F_\alpha \cap U} \in \Phi^f_{i, F_\alpha \cap U}$ and $h^{f}_{i+1, F_\alpha \cap U} \in \Phi^f_{i+1, F_\alpha \cap U}$ with the property 
$dom(h^{f}_{i+1, F_\alpha \cap U}) \subseteq dom(h^{f}_{i, F_\alpha \cap U}) $ we have 
$$h^{f}_{i, F_\alpha \cap U} (\beta) >h^{f}_{i+1, F_\alpha \cap U}(\beta)$$ for all $\beta \in dom(h^{f}_{i+1, F_\alpha \cap U}) $ and $i = 0, 1, ...\ $.

Now,  take $\Phi^f_{i} =\bigcup_{(F_\alpha \cap U) \in \mathcal{F}_f} \Phi^f_{i, F_\alpha \cap U}, $ for each $i = 0, 1, ...$\ .
Obviously $\Phi^f_{0} > \Phi^f_{1} >...$ is a sequence of functionals on $S^f = \bigcup_{(F_\alpha \cap U) \in \mathcal{F}_f} S^f_{F_\alpha \cap U}\in P(\kappa)\setminus I_\mathcal{F}$ and $W^f_{i} = \bigcup_{(F_\alpha \cap U) \in \mathcal{F}_f} W^f_{i, F_\alpha \cap }$.
For each $i = 0, 1, ...$ consider $h^f_i = \bigcup_{(F_\alpha \cap U) \in \mathcal{F}_f}h^{f}_{i, F_\alpha \cap U}$. Obviously 
$dom(h^f_{i})$ is comeager and $$h^f_{i}(\beta) > h^f_{i+1} (\beta)$$ for all $\beta \in dom(h^f_{i+1})$ and $i = 0, 1, ...\ $.

Now, for each $i = 0, 1, ...$ consider $\Phi_i = \bigcup_{f \in FN(\kappa)} \Phi^f_i$. Repeating the adequate considerations as for $\Phi^f_i$ we obtain the sequences of functions 
$h_i = \bigcup_{f \in FN(\kappa)} h^f_i, i = 0, 1, ...$ such that 
$$h_{i}(\beta) > h_{i+1} (\beta)$$
for all $\beta \in dom(h_{i+1})$.

For each such a function $h_i$ and each $\beta \in dom (h_i)$ take $f_i \in FN(\kappa)$ such that 
$$h_i(\beta) = f_i (x),$$
where $x \in F_\beta \cap U$ for some $F_\beta \cap U \in \mathcal{F}_{f_i}, i = 0, 1, ... \ .$
By the Baire Category Theorem there exists $x \in X$ such that $f_{0}(x) > f_{1}(x) >...\ $. A contradiction with the well-foundness of the sequence. 
\end{proof}
 \\
 
As was shown above, if a space has a Kuratowski partition then the ideal associated with such a partition can be precipitous. As there is a difference between a complete metric metric space and a Baire metric space, in the \cite{FJ} we show that if there is a Kuratowski partition of a complete metric space then there is a measurable cardinal.

{\sc Ryszard Frankiewicz}
\\
Institute of Mathematics, Polish Academy of Sciences, Warsaw, Poland,
\\
{\sl e-mail: rf@impan.pl}
\\

{\sc Joanna Jureczko}
\\
Wroc\l{}aw University of Science and Technology, Wroc\l{}aw, Poland,
\\
{\sl e-mail: joanna.jureczko@pwr.edu.pl}

\end{document}